\documentclass[10pt]{article}
\usepackage{BICA}
\usepackage{amsmath,amsthm,amsfonts,amssymb,url}
\newtheorem{theorem}{Theorem}[section]
\newtheorem{lemma}[theorem]{Lemma}
\newtheorem{corollary}[theorem]{Corollary}
\newtheorem{remark}{Remark}

\newtheorem{example}{Example}

\newcommand{\OA}{\ensuremath{\mathsf{OA}}} 
\newcommand{\B}{\ensuremath{\mathcal{B}}} 
\newcommand{\G}{\ensuremath{\mathcal{G}}} 
\newcommand{\zed}{\ensuremath{\mathbb{Z}}}

\title{Bounds for orthogonal arrays with repeated rows}
\author{Douglas~R.~Stinson\thanks{The author's research is supported by  
NSERC discovery grant RGPIN-03882.}\\
David R.\ Cheriton School of Computer Science\\ University of Waterloo\\
Waterloo, Ontario N2L 3G1, Canada\\
{\tt dstinson@uwaterloo.ca}}
\date{}

\begin{document}
\maketitle

\begin{abstract}
In this expository paper, we mainly study orthogonal arrays (OAs) 
of strength two having a row that is repeated $m$ times. 
It turns out that the Plackett-Burman bound (\cite{PB}) can be strengthened by a factor of $m$ for orthogonal arrays 
of strength two that contain a row that is repeated $m$ times. This is a consequence of 
a more general result due to Mukerjee,  Qian and  Wu \cite{Muk} that applies to orthogonal
arrays of arbitrary strength $t$.

We examine several proofs
of the Plackett-Burman bound and discuss which of these proofs can be strengthened to yield
the aforementioned bound for OAs of strength two with repeated rows. We also briefly discuss
related bounds for $t$-designs, and OAs of strength $t$, when $t > 2$.
\end{abstract}

\section{Introduction}

Balanced incomplete block designs (BIBDs) with repeated blocks have been studied for many years.
A classical result due to Mann \cite{Mann} in 1969 says that a BIBD having a block of multiplicity $m$
satisfies the inequality $b \geq mv$, where $b$ is the number of blocks and $v$ is the
number of points. This is of course a generalization of Fisher's inequality, which states
that $b \geq v$ in any BIBD.

In this expository paper, we mainly study the corresponding problem for orthogonal arrays (OAs) 
of strength two. Here, the analog of Fisher's inequality is the Plackett-Burman bound \cite{PB}, and it
turns out that this bound can be strengthened by a factor of $m$ for orthogonal arrays 
of strength two that contain a row that is repeated $m$ times. This is in fact a consequence of 
a more general result due to Mukerjee,  Qian and  Wu \cite{Muk} that applies to orthogonal
arrays of arbitrary strength $t$ (see Theorem \ref{muk.thm}).

Aside from being an interesting theoretical question, there is also some practical motivation for studying this problem.
In a recent paper, Culus and  Toulouse \cite{CT} discuss an application where it is 
beneficial to construct orthogonal arrays with an $m$-times repeated row, where 
$m$ is as large as possible.

The rest of this paper is organized as follows.
In Section \ref{OA2.sec}, we present a standard proof of the Plackett-Burman bound and 
show how it can be easily modified to yield the desired inequality for OAs having repeated rows
(this is Theorem \ref{repeated.thm}).
In Section \ref{other.sec}, we examine several other proofs
of the Plackett-Burman bound and discuss which of these proofs can be strengthened to yield
the aforementioned bound for OAs of strength two with repeated rows. 
In Section \ref{t>2.sec}, we discuss
related bounds for OAs of strength $t$, when $t > 2$.
Finally, in Section  \ref{t-designs.sec}, we consider 
related bounds for balanced incomplete block designs (BIBDs)  and $t$-designs.

This paper is exclusively concerned with proving bounds. A follow-up paper \cite{CSV}
investigates constructions 
for orthogonal arrays of strength two, having an $m$-times repeated row, that meet the bound
proven in Theorem \ref{repeated.thm} with equality.

\section{Orthogonal arrays of strength 2}
\label{OA2.sec}

Let $k \geq 2$, $n \geq 2$ and $\lambda \geq 1$ be integers.
An \emph{orthogonal array} $\OA_{\lambda} (k,n)$
is a $\lambda  n^2 $ by $k$ array, $A$, with entries from a set
$X$ of cardinality $n$ such that, within any two columns of $A$,
every ordered pair of symbols from $X$ occurs in exactly $\lambda$ rows of $A$.
We often denote the number of rows, $\lambda  n^2$, by $N$.
For much information on orthogonal arrays, see \cite{HSS}.

We first give a standard combinatorial proof of the classical Plackett-Burman bound from 1946.
This proof is similar to the original proof of Fisher's inequality \cite{Fisher}.
The proof technique is sometimes called the ``variance method''.

\begin{theorem}[Plackett-Burman bound (\cite{PB})]
\label{PB.thm}
Let $k \geq 2$, $n \geq 2$ and $\lambda \geq 1$ be integers. 
If there is an $\OA_{\lambda}(k,n)$, then
\[\lambda \geq \frac{k(n-1)+1}{n^2}.\]
\end{theorem}

\begin{proof}
Relabel the symbols in each column of the OA so the last row of $A$ is $1 \: 1\: \cdots \: 1$.
For $1 \leq i \leq N-1$, let $a_i$ denote the number of
``$1$''s in row $i$ of $A$.
Then
\begin{eqnarray*}
\sum _{i=1}^{N-1} a_i &=& k(\lambda n - 1), \\
\sum _{i=1}^{N-1} a_i(a_i-1) &=& k(k-1)(\lambda  - 1), \quad \text{and} \\
\sum _{i=1}^{N-1} {a_i}^2 &=& k(k(\lambda - 1) + \lambda(n-1)) .
\end{eqnarray*}
Define 
\[ \overline{a} = \frac{k(\lambda n - 1)}{N-1}.\]
Then
\begin{eqnarray*}
0 & \leq & \sum _{i=1}^{N-1} (a_i - \overline{a})^2 \\
&=& \sum _{i=1}^{N-1} {a_i}^2 - 2 \overline{a} \sum _{i=1}^{N-1} a_i +
\overline{a}^2 (N-1) \\
&=& k(k(\lambda - 1) + \lambda(n-1)) - \frac{k^2(\lambda n -1)^2}{\lambda n^2 - 1}.
\end{eqnarray*}
Therefore, 
\[ k (\lambda n - 1)^2 \leq (\lambda n^2 - 1)(k(\lambda - 1) + \lambda(n-1)).\]
This simplifies to yield
\[k \leq \frac{\lambda n^2 - 1}{n-1}.\]
Equivalently, 
\[\lambda \geq \frac{k(n-1)+1}{n^2}.\]
\end{proof}

Using the identical proof technique, we also have the following theorem.

\begin{theorem}
\label{repeated.thm}
Let $k \geq 2$, $n \geq 2$ and $\lambda \geq 1$ be integers.
If there is an $\OA_{\lambda}(k,n)$ containing a row that is repeated $m$ times, then
\[\lambda \geq \frac{m(k(n-1)+1)}{n^2}.\]
\end{theorem}

\begin{proof}
Reorder the rows so the last $m$ rows are identical. 
Then relabel the symbols in each column so the last  $m$ rows of $A$ are $1 \: 1\: \cdots \: 1$.
For $1 \leq i \leq N-m$, let $a_i$ denote the number of
``$1$''s in row $i$ of $A$.
Then
\begin{eqnarray*}
\sum _{i=1}^{N-m} a_i &=& k(\lambda n - m),  \\
\sum _{i=1}^{N-m} a_i(a_i-1) &=& k(k-1)(\lambda  - m), \quad \text{and} \\
\sum _{i=1}^{N-m} {a_i}^2 &=& k(k(\lambda - m) + \lambda(n-1)) .
\end{eqnarray*}
Define 
\[ \overline{a} = \frac{k(\lambda n - m)}{N-m}.\]
Then
\begin{eqnarray*}
0 & \leq & \sum _{i=1}^{N-m} (a_i - \overline{a})^2 \\
&=& \sum _{i=1}^{N-m} {a_i}^2 - 2 \overline{a} \sum _{i=1}^{N-m} a_i +
\overline{a}^2 (N-m) \\
&=& k(k(\lambda - m) + \lambda(n-1)) - \frac{k^2(\lambda n -m)^2}{\lambda n^2 - m}.
\end{eqnarray*}
Therefore, 
\[ k (\lambda n - m)^2 \leq (\lambda n^2 - m)(k(\lambda - m) + \lambda(n-1)).\]
This simplifies to yield
\[k \leq \frac{\lambda n^2 - m}{m(n-1)}.\]
Equivalently, 
\[\lambda \geq \frac{m(k(n-1)+1)}{n^2}.\]
\end{proof}


The following corollary is immediate.

\begin{corollary}
\label{repeated.cor}
Let $k \geq 2$, $n \geq 2$ and $\lambda \geq 1$ be integers.
If there is an $\OA_{\lambda}(k,n)$ containing a row that is repeated $m$ times, then
\[m \leq \frac{\lambda n^2}{k(n-1)+1}.\]
\end{corollary}

\begin{remark}
The total number of rows in the orthogonal array is $N = \lambda n^2$.
Thus the inequality from Corollary \ref{repeated.cor} can also be written as
\[m \leq \frac{N}{k(n-1)+1}.\]
\end{remark}

\begin{example}
Suppose we take 
$n=2$, $k=4$ and $\lambda = 2$ in Corollary \ref{repeated.cor}.
Then
\[m \leq \frac{8}{4(2-1)+1} = \frac{8}{5}.\]
Since $m$ is an integer, an $\OA_2(4,2)$ cannot have any repeated rows.

If we take $n=2$, $k=4$ and $\lambda = 3$ in Corollary \ref{repeated.cor},
then we see that
\[m \leq \frac{12}{4(2-1)+1} = \frac{12}{5}.\]
Since $m$ is an integer, an $\OA_3(4,2)$ cannot contain a row that is repeated three times.
\end{example}

\begin{example}
An $\OA_3(5,3)$ is presented in \cite[Table 2]{CT}.
This orthogonal array contains a row that is repeated twice.
If we take $n=3$, $k=5$ and $\lambda = 3$ in Corollary \ref{repeated.cor},
we see that
\[m \leq \frac{27}{5(3-1)+1} = \frac{27}{11}.\]
Therefore, an $\OA_3(5,3)$  cannot contain a row that is repeated three times.
\end{example}

We close this section by observing that the case of equality in the bound proven in
Theorem \ref{repeated.thm} can easily be characterized.

\begin{corollary}
\label{abar.lem}
Let $k \geq 2$, $n \geq 2$ and $\lambda \geq 1$ be integers.
Suppose there is an  $\OA_{\lambda}(k,n)$, say $A$, containing a row 
$1 \; 1 \cdots 1$ that is repeated $m$ times, where 
$m = \lambda n^2/(k(n-1)+1).$
Then every other row of $A$ contains exactly 
\begin{equation}
\label{a.eq}
\overline{a} = \frac{k(\lambda n - m)}{\lambda n^2-m}
\end{equation} 
occurrences of the symbol $1$.
\end{corollary}

\begin{proof}
In the proof of Theorem \ref{repeated.thm}, in order to obtain equality in the resulting
bound, we must have 
\[ \sum _{i=1}^{N-m} (a_i - \overline{a})^2 = 0,\] 
so $\overline{a} = a_i$ for all $i$.
\end{proof}

\section{Some other proofs}
\label{other.sec}

In this section, we consider other proofs of the Plackett-Burman bound
and investigate whether they can be modified to prove Theorem \ref{repeated.thm}.

\subsection{Constant weight codes}

It is not hard to see that the Plackett-Burman bound and  Theorem \ref{repeated.thm} can be both 
derived as an immediate corollary of a
1962 bound of Johnson \cite[Theorem 3]{Johnson} for constant-weight binary codes. 
Johnson defines $R(\ell,r,\mu)$ to be the maximum number of $0-1$ vectors of length $\ell$
and weight $w$ such that the inner product of any two of the vectors is at most $\mu$. Then he proves the following bound.

\begin{theorem}
\cite[Theorem 3]{Johnson}
\label{johnson.thm}
If $w^2 > \ell \mu$, then
\[ R(\ell,w,\mu) \leq \frac{\ell(w - \mu)}{w^2 - \ell \mu} .\]
\end{theorem}

We describe how to derive Theorem \ref{repeated.thm}  as a corollary of Theorem \ref{johnson.thm}.
Suppose that $A$ is  an $\OA_{\lambda}(k,n)$ in which the last $m$ rows are $1 \: 1\: \cdots \: 1$.
Delete these $m$ rows and replace every symbol other than $1$ by $0$.
Consider the resulting set of $k$ column vectors of length $\lambda n^2 - m$.
Each of these vectors has weight $\lambda n - m$ and the inner product of any two of these vectors
is equal to $\lambda - m$. From Theorem \ref{johnson.thm}, we see that
\[
 k \leq \frac{(\lambda n^2 - m)\lambda(n-1)}{(\lambda n - m)^2 - (\lambda n^2 - m)(\lambda - m)} 
=\frac{\lambda n^2 - m}{m(n-1)}.
 \]

The above derivation might appear to be simpler than the one we gave in the proof of Theorem \ref{repeated.thm}.
But this is illusory, as the proof of Theorem \ref{johnson.thm} relies on the computation
of a quadratic sum analogous to the one considered in the proof of Theorem \ref{repeated.thm}.

\subsection{Transversal designs}
\label{TD.sec}

Another way to prove the Plackett-Burman bound involves constructing the incidence matrix, $M$, of the
associated transversal design. (For the definition of a transversal design, and
a description of how to transform an orthogonal array into a transversal design,
see \cite[\S 6.6]{Stinson}.) The incidence matrix $M$ has $kn$ columns (corresponding to the $kn$ points in
the transversal design) and $k + \lambda n^2$ rows (where the first $\lambda n^2$ rows correspond to the rows of the OA, 
and the last $k$ rows are associated with the groups of the transversal design). 

It is fairly straightforward 
to show that the rows of $M$ span a real $kn$-dimensional vector space, whence
$k + \lambda n^2 \geq kn$. Then we can observe that one of the rows of $M$ is spanned by the remaining rows,
so the inequality can be strengthened to $k + \lambda n^2 -1 \geq kn$, which is just the Plackett-Burman bound.

Here are the details required to fill in the proof.  Let the transversal design be $(X,\G,\B)$ where
\begin{eqnarray*}
X &=& \{1, \dots , n\} \times \{1, \dots , k\},\\
\G &=& \{G_1, \dots , G_k\},\\
G_j &=& \{1, \dots , n\} \times \{j\},   \text{for $1 \leq j \leq k$, and}\\
\B &=& \{B_1, \dots , B_{N}\}, \text{where $N = \lambda n^2$.}
\end{eqnarray*}
Members of $X$ are termed \emph{points},
 members of $\G$ are called \emph{groups} and  members of $\B$ are referred to as \emph{blocks}.

The incidence matrix $M$ is defined as follows. A row of $M$ is indexed by a block or a group and
a column of $M$ is indexed by a point $x$. The corresponding entry of $M$ is ``1'' if $x$ is a member
of the given block or group, and ``0'' otherwise. 

Each row of the incidence matrix $M$ is a $0-1$ vector of length $kn$. We thus obtain
$k + N$ vectors, which we will denote by 
$\mathbf{B_i}$ ($1 \leq i \leq N$) and $\mathbf{G_j}$ ($1 \leq j \leq k$).
For any point $x \in X$, let $\mathbf{x}$ denote the standard basis vector that has a ``1'' in the
co-ordinate corresponding to $x$. The all-ones vector will be denoted by $\mathbf{u}$.

We list a few obvious equations:
\begin{equation}
\label{eq1}
\sum _{i = 1}^{N} \mathbf{B_i} = \lambda n \mathbf{u}.
\end{equation}
\begin{equation}
\label{eq2}
\sum _{j = 1}^{k} \mathbf{G_j} =  \mathbf{u}.
\end{equation}
For any $x \in G_j$, it holds that
\begin{equation}
\label{eq3}
\lambda \mathbf{G_j} + \sum _{\{i : x \in B_i\}} \mathbf{B_i}  = \lambda \mathbf{u} + (\lambda  n)\mathbf{x}.
\end{equation}

From (\ref{eq1}) and (\ref{eq3}), it is easy to see that
\[\mathbf{x} \in \mathsf{Span} (\mathbf{B_1}, \dots , \mathbf{B_N}, \mathbf{G_1}, \dots , \mathbf{G_k})\]
for any $x \in X$. Then, using (\ref{eq1}) and (\ref{eq2}), it follows that
\[\mathbf{G_k} \in \mathsf{Span} (\mathbf{B_1}, \dots , \mathbf{B_N}, \mathbf{G_1}, \dots , \mathbf{G_{k-1}}).\]
Therefore, 
\[\mathbf{x} \in \mathsf{Span} (\mathbf{B_1}, \dots , \mathbf{B_N}, \mathbf{G_1}, \dots , \mathbf{G_{k-1}})\]
for any $x \in X$. This immediately implies
$k + N -1 \geq kn$, and the Plackett-Burman bound is proven.

\begin{remark}
The above proof is obviously motivated by a similar proof of Fisher's inequality, which can
be found, for example, in \cite[Theorem 1.33]{Stinson}. This basic method, as far as I can determine,
was first used by Ray-Chaudhuri and Wilson in the proof of \cite[Theorem 1]{RCM}.
\end{remark}

\begin{remark}
\label{weak.rem}
Interestingly, there does not seem to an obvious way to modify this proof to yield Theorem \ref{repeated.thm}.
Cleary, if the OA contains $m$ identical rows, then the associated incidence matrix contains $m$ identical rows.
Suppose we delete $m-1$ of the $m$ identical rows in $M$, producing a matrix $M'$. The span of the rows of $M'$ 
is identical to the span of the rows of $M$, from which it immediately follows that
$k + N - m \geq kn$. But this is a much weaker inequality than the one proven in Theorem \ref{repeated.thm}.
\end{remark}

\subsection{Evaluating a determinant}

We consider a slightly modified incidence matrix:
\begin{enumerate}
\item Transpose the $N+ k$ by $nk$ matrix
$M$ defined in Section \ref{TD.sec}. 
\item Adjoin an additional row to the resulting matrix, which contains
0's in the first $N$ columns (which are labelled by blocks) and 1's in the 
last $k$ columns (which are labelled by groups).
\item Multiply all elements in the last $k$ columns by $\sqrt{\lambda}$.
Call the resulting $nk+1$ by $N+k$ matrix $\tilde{M}$.
\end{enumerate}

The following result can be easily verified.

\begin{lemma} The matrix $\tilde{M}$, as defined above, satisfies the matrix equation
\[ \tilde{M}\tilde{M}^T = \lambda J + \mathsf{diag}(\lambda n , \dots , \lambda n , (k-1)\lambda),\]
where $J$ is the all-ones matrix and $\mathsf{diag}()$ is a matrix with the specified entries on
the diagonal and 0's elsewhere.
\end{lemma}

The Plackett-Burman bound follows by proving that $\det (\tilde{M}\tilde{M}^T) > 0$.
(There are various ways to do this, but we do not go into the details here.)
Then, because  $\tilde{M}\tilde{M}^T$ is an $nk+1$ by $nk+1$ matrix, we have
\[ nk+1 = \mathsf{rank}(\tilde{M}\tilde{M}^T) \leq \mathsf{rank}(\tilde{M}) \leq \min \{ nk+1,N+k\} \leq N+k.\]

\begin{remark}
This proof is again a modification of a familiar proof of Fisher's inequality, namely, the
proof given by Bose \cite{Bose}.
It might in fact be possible to extend this proof to derive Theorem \ref{repeated.thm}, by employing
 Connor's inequalities \cite{Connor}.
\end{remark}

\subsection{Orthogonal vectors}

The last proof of the Plackett-Burman bound that we present is a specialization
of a standard proof of the Rao bound for orthogonal arrays of
arbitrary strength $t$. Assume that $A$ is an $\OA_{\lambda}(k,n)$
defined on the symbol set $\zed_n$. Let the 
columns of $A$ be denoted $C_1, \dots , C_k$. Let $C_0$ be the column
vector of ``$0$''s.
For $1 \leq m \leq n-1$, construct $m C_j$ 
from $C_j$ by multiplying every entry by $m$ (modulo $n$).
Let $\omega = e^{2 \pi i / n}$ (a complex $n$th root of 1) and define
$\phi : \zed_n  \rightarrow \mathbb{C}$ by the rule
$\phi (s) = \omega ^s$.

Now consider the set of $1 + k(n-1)$ vectors
\[ \mathcal{D} =  \{ \phi(C_0)\} \cup  
\{ \phi(m C_j) : 1 \leq m \leq n-1, 1 \leq j \leq k \} .\]
It can be shown easily that $\langle C,D \rangle = 0$ for all $C,D \in \mathcal{D}$,
$C \neq D$, where $\langle \cdot , \cdot \rangle $ denotes the hermitian 
inner product of two (complex-valued) vectors. (This uses the fact that the sum of the 
complex $n$th roots of $1$ equals $0$.)

Since $\mathcal{D}$ consists of mutually orthogonal vectors, 
they are linearly independent. Hence, we have a set of
$1 + k(n-1)$ linearly independent vectors in $\mathbb{C}^N$,
and it follows that $1+ k(n-1) \leq N (= \lambda n^2)$.

\begin{remark}
There does not seem to be any obvious way to modify this proof
to obtain Theorem \ref{repeated.thm}. But it is possible to obtain by this method
the weaker bound stated in Remark \ref{weak.rem}. Suppose that the last $m$ rows of $A$ 
are all $0$'s. Then, in any vector in $\mathcal{D}$, the last $m$ coordinates are identical.
Suppose that we delete the last $m-1$ co-ordinates of each of these vectors, and then multiply the last
(remaining) co-ordinate by $\sqrt{m}$. This will have the effect of preserving the condition $\langle C,D \rangle = 0$
for all the shortened vectors. We end up with $1 + k(n-1)$ linearly independent vectors in $\mathbb{C}^{N-m+1}$.
Therefore, $1+ k(n-1) \leq N-m+1$ and hence $k + N - m \geq kn$.
\end{remark}

\section{OAs of strength $t > 2$}
\label{t>2.sec}

We now consider a generalization. Let $t \geq 2$ be a parameter called the \emph{strength}, and as
before let $k\geq 2$, 
$n \geq 2$ and $\lambda \geq 1$ be integers.
An \emph{orthogonal array} $\OA_{\lambda} (t,k,n)$
is a $\lambda  n^t $ by  $k$ array, $A$, with entries from a set
$X$ of cardinality $n$ such that, within any $t$ columns of $A$,
every $t$-tuple of symbols from $X$ occurs in exactly $\lambda$ rows of $A$.

Observe that an $\OA_{\lambda} (2,k,n)$ is the same thing as an 
$\OA_{\lambda} (k,n)$.  One main bound for $\OA_{\lambda} (2,k,n)$
is the Rao bound from 1947.

\begin{theorem}[Rao bound (\cite{Rao})]
\label{rao.thm}
Let $k \geq 2$, $t \geq 2$, $n \geq 2$ and $\lambda \geq 1$ be integers.
If there is an $\OA_{\lambda}(t, k,n)$, then
\[ \lambda n^t \geq 
\begin{cases}
1+ \sum_{i=1}^{t/2} \binom{k}{i}(n-1)^i & \mbox{if $t$ is even}\\
1+ \sum_{i=1}^{(t-1)/2} \binom{k}{i}(n-1)^i + 
\binom{k-1}{(t-1)/2}(n-1)^{(t+1)/2} & \mbox{if $t$ is odd.}
\end{cases}\]
\end{theorem}

The Plackett-Burman bound (Theorem \ref{PB.thm}) 
is just the special case of the Rao bound with $t=2$.

We are of course interested in the situation where the orthogonal array contains
a row that is repeated $m$ times. We might expect that the Rao bound could 
be improved by a multiplicative factor of $m$ in this case. Indeed, a more
general result of  Mukerjee,  Qian and  Wu \cite{Muk} yields the desired inequality.
The paper \cite{Muk} defines a {\it nested orthogonal array}, which refers to an
$\OA_{\lambda} (t,k,n)$ that contains a subset of rows that is 
an $\OA_{\lambda'} (t,k,n')$. If $n' = 1$ and $\lambda' = m$, then we just have
an $m$-times repeated row in the $\OA_{\lambda} (t,k,n)$.
A necessary condition for the existence of nested orthogonal arrays is
proven in \cite[Theorem 1]{Muk}. Specializing this result to orthogonal
arrays that contain an $m$-times repeated row, we obtain the following:

\begin{theorem}[Mukerjee-Qian-Wu bound (\cite{Muk})]
\label{muk.thm}
Let $k \geq 2$, $t \geq 2$, $n \geq 2$ and $\lambda \geq 1$ be integers.
If there is an $\OA_{\lambda}(t, k,n)$ containing a row that is repeated $m$ times, then
\[ \lambda n^t \geq 
\begin{cases}
m\left(1+ \sum_{i=1}^{t/2} \binom{k}{i}(n-1)^i\right) & \mbox{if $t$ is even}\\
m\left(1+ \sum_{i=1}^{(t-1)/2} \binom{k}{i}(n-1)^i + 
\binom{k-1}{(t-1)/2}(n-1)^{(t+1)/2}\right) & \mbox{if $t$ is odd.}
\end{cases}\]
\end{theorem}

The proof of the Rao bound (Theorem \ref{rao.thm}) is not difficult. Unfortunately, the proof of
\cite[Theorem 1]{Muk} is much more complicated. It is remarked in \cite{Muk} that there is no obvious
way to generalize the standard proof of the Rao bound to nested orthogonal arrays.

\section{BIBDs and $t$-designs}
\label{t-designs.sec}

So far, we have concentrated on bounds for
orthogonal arrays.  It may be of interest to look at ``analogous'' 
inequalities for BIBDs (i.e., $2$-designs) and $t$-designs
for general values of $t\geq 2$. In our discussion,  
$b$ will denote the number of blocks in a $t$-$(v,k,\lambda)$-design.

The following summarizes some of the main known results.

\begin{enumerate}
\item Fisher's inequality 
for $2$-designs (proven in 1940; see \cite{Fisher}) states that if a $2$-$(v,k,\lambda)$-design  exists
with $v \geq k+1$, then 
$b \geq v$.  There are by now many proofs of this inequality, several of which have been
mentioned in Sections \ref{OA2.sec} and \ref{other.sec}.
\item The Ray-Chaudhuri and Wilson inequality \cite{RCM} asserts that 
if a $t$-$(v,k,\lambda)$-design exists with $t\geq 2s$ and $v \geq k+s$, then
$b \geq \binom{v}{s}$. (This result of course generalizes Fisher's inequality, and we also 
remark that the case $t=4$ was first proven by Petrenjuk \cite{Pet}, who conjectured the inequality
proven by Ray-Chaudhuri and Wilson.) As we mentioned earlier, the proof given in \cite{RCM} uses a technique
similar to the one in Section \ref{TD.sec}. Some other proofs can be found in 
\cite[Theorem 1]{Wil2}, \cite[Theorem 1.4.1]{Godsil} and \cite[Theorem 3.4]{MS}.
\item Mann's inequality \cite{Mann}, which was proven in 1969, states that if a $2$-$(v,k,\lambda)$-design 
with $v \geq k+1$ has a block that is repeated $m$ times, then $b \geq mv$.
(The special case $m=1$ is just Fisher's inequality). Mann's proof was an application of the
variance method. Interestingly, a slightly weaker inequality was proven a few years earlier,
using the same method, by Stanton and Sprott \cite{StSp}.

For another proof of Mann's inequality, see
\cite{vLR}. It is also well-known that Mann's inequality can be derived as a consequence of 
Connor's inequalities  \cite{Connor}, which were proven in 1952; see, for example, Brouwer \cite[p.\ 700]{Brouwer}.
\item Finally, Wilson (\cite[Corollary 3]{Wil2} and \cite[Corollary 4]{Wil}) proved that 
if there exists a $t$-$(v,k,\lambda)$-design with $t\geq 2s$ 
and $v \geq k+s$, which has a block that is repeated $m$ times, then $b \geq m\binom{v}{s}$.
Wilson's proofs use methods of orthogonal projections.
For another proof of this result, based on the so-called ``cone condition'', see \cite{DD}.

This result of course immediately implies the three previous results.
However, its proofs are substantially more complicated than the proofs of the three previous results.
\end{enumerate}

\section*{Acknowledgement}
I would like to thank Sophie Toulouse for bringing this problem to my attention.
Thanks also go to the referees for helpful comments and suggestions.


\begin{thebibliography}{X}

\bibitem{Bose}
R.C.\ Bose.
A note on Fisher's inequality for balanced incomplete block designs.
{\it Annals of Mathematical Statistics}
{\bf 20} (1949), 619--620.

\bibitem{CSV}
C.J.\ Colbourn, D.R.\ Stinson and S.\ Veitch.
Constructions of optimal orthogonal arrays with repeated rows.
Preprint.


\bibitem{Brouwer}
A.E.\ Brouwer.
{\it Block Designs}. Chapter 14 in 
``Handbook of Combinatorics'', R.L.\ Graham, ed., North Holland, 1995.

\bibitem{Connor}
W.S.\ Connor, Jr. 
On the structure of balanced incomplete block designs.
{\it Annals of Mathematical Statistics}
{\bf 23} (1952), 57--71.

\bibitem{CT}
J.-F.\ Culus  and S.\ Toulouse.
How far from a worst solution
a random solution of a $k$ CSP
instance can be?
{\it Lecture Notes in Computer Science} {\bf 10979} (2018), 374--386
(IWOCA 2018).

\bibitem{DD}
P.J.\ Dukes and R.M.\ Wilson. 
The cone condition and $t$-designs. {\it European J.\ Combin.} {\bf 28} (2007), 
1610--1625.

\bibitem{Fisher}
R.A.\ Fisher.
An examination of the different possible solutions of a problem in incomplete blocks.
{\it Annals of Eugenics} {\bf 10} (1940), 52--75.

\bibitem{Godsil}
C.D.\ Godsil.
{\it Linear Algebra and Designs}.
Lecture notes, 1995.

\bibitem{HSS}
A.S.\ Hedayat, N.J.A.\ Sloane and John Stufken.
{\it Orthogonal Arrays,
Theory and Applications.}
Springer, 1999.

\bibitem{Johnson}
S.M.\ Johnson.
A new upper bound for error-correcting codes.
{\it IRE Transactions on Information Theory} {\bf 8} 
(1962), 203--207.



\bibitem{Mann}
H.B.\ Mann.
A note on balanced incomplete-block designs.
{\it Ann.\ Math Statist.} {\bf 40} (1969), 679--680.

\bibitem{MS}
W.J.\ Martin and D.R.\ Stinson.
A polynomial ideal associated to any $t$-$(v, k, \lambda)$ design.
\url{arXiv:1803.05004}.

\bibitem{Muk}
Rahul Mukerjee, Peter Qian and Jeff Wu.
On the existence of nested orthogonal arrays.
{\it Discrete Math.} {\bf 308} (2008), 4635--4642.

\bibitem{Pet} A.Ja.\ Petrenjuk. On Fisher's inequality for tactical configurations (Russian), 
{\it Mat.\
Zametki} {\bf 4} (1968), 417--425.

\bibitem{PB}
R.L.\ Plackett and J.P.\ Burman.
The design of optimum multifactorial experiments.
{\it Biometrika}
{\bf 33} (1946), 305--325.

\bibitem{Rao}
C.R.\ Rao. 
Factorial experiments derivable from combinatorial arrangements of arrays. 
{\it J.\ Royal Statist.\ Soc.\ (Suppl.)} {\bf 9} (1947), 128--139.

\bibitem{RCM}
D.K.\ Ray-Chaudhuri and R.M.\ Wilson.
On $t$-designs. {\it Osaka Journal of Mathematics} {\bf 12} (1975), 737--744.

\bibitem{StSp}
R.G.\ Stanton and D.A.\ Sprott.
Block intersections in balanced incomplete block designs.
{\it Canadian Math. Bulletin} {\bf 7} (1964), 539--548.

\bibitem{Stinson}
D.R.\ Stinson. {\it Combinatorial Designs, Constructions and Analysis}, Springer, 2004.

\bibitem{vLR}
J.H.\ van Lint and 
H.J.\ Ryser.
Block designs with repeated blocks.
{\it Discrete Mathematics} {\bf 3} (1972), 381--396.

\bibitem{Wil2}
R.M.\ Wilson.
Incidence matrices of $t$-designs. {\it Linear Algebra and its Applications} {\bf 46} (1982), 73--82.

\bibitem{Wil}
R.M.\ Wilson.
Inequalities for  $t$-designs. {\it Journal of Combinatorial Theory A} {\bf 34} (1983), 313--324.


\end{thebibliography}
\end{document}